\documentclass[letterpaper, oneside]{amsart}

\usepackage{layout}	

\usepackage[
]{geometry}

\usepackage[parfill]{parskip}   
\usepackage{diagbox}
									
\usepackage{amssymb}
\usepackage{amsthm}
\usepackage{amsmath}
\usepackage{cite}
\usepackage[all]{xy}
\usepackage{enumerate}
\usepackage{verbatim}
\usepackage{mathtools}
\usepackage{hyperref}
\usepackage{colonequals}
\usepackage{mathrsfs}
\usepackage{color}

\usepackage{tikz}
\usetikzlibrary{decorations.markings}
\theoremstyle{plane}
\newtheorem{thm}{Theorem}[section]
\newtheorem{lem}[thm]{Lemma}

\newtheorem{cor}[thm]{Corollary}

\theoremstyle{definition}

\newtheorem*{rmk}{Remark}
\newtheorem*{ack}{Acknowledgement}

\newtheorem{example}[thm]{Example}

\newcommand{\Z}{\mathbb{Z}}
\newcommand{\N}{\mathbb{N}}

\newcommand{\C}{\mathbb{C}}

\newcommand{\B}{\mathcal{B}}

\newcommand{\floor}[1]{\left\lfloor#1\right\rfloor}

\title{Stability of Springer representations for type $A$}
\author{Dongkwan Kim}
\address{Department of Mathematics\\
  Massachusetts Institute of Technology\\
  Cambridge, MA 02139-4307\\
  U.S.A.}
\email{sylvaner@math.mit.edu}
\date{\today}							% Activate to display a given date or no date

\begin{document}
\begin{abstract} 
We prove certain stability properties of Springer representations for type $A$.
\end{abstract}

\maketitle

\renewcommand\contentsname{}
\tableofcontents

\section{Introduction}
In his paper \cite{spr}, Springer defined a certain representation of the Weyl group of a reductive group on the cohomology of a set of fixed points in the flag variety by a nilpotent element, which we now call the Springer representation. In particular, for various sets of such fixed points, called Springer fibers, the top degree cohomology gives all the irreducible representations of the corresponding Weyl group. 

Consider Springer representations of type $A_{n-1}$, i.e. of the symmetric group of $n$ elements. In this case, if we choose a nilpotent element parametrized by a partition $\lambda$, then it is known that the top degree cohomology of the Springer fiber corresponding to this nilpotent element gives the irreducible representation of the symmetric group also parametrized by $\lambda$.

In this paper we prove some stability conditions of Springer representations of type $A$. To be precise, we prove that if some partitions $\lambda, \lambda'$ of $n$ satisfy certain conditions, then the cohomology of the corresponding Springer fibers on each degree $k$ give the same representation of the symmetric group. Also we prove that there exists a polynomial, depending on the degree $k$, such that the $k$-th Betti number of such Springer fiber is the value of this polynomial evaluated at $n$.

There exists a notion of stability defined by \cite{c-f} and \cite{cef}. In this language \cite{mt} proved another version of stability conditions of Springer representations. Note that our result is related to but different from \cite{mt}. We remark the connection between our result and \cite{mt} at the end of this paper.

\begin{ack} The author thanks George Lusztig for useful comments. In particular, monotonicity property of Kostka-Foulkes polynomials was kindly informed to the author by him. 
\end{ack}

\section{Notations and preliminaries}
We denote by $\B_n$ the flag variety of $GL_n(\C)$, i.e. the set of complete flags in $\C^n$. We define $W_n$ to be the symmetric group of $n$ elements, which is isomorphic to the Weyl group of $GL_n(\C)$.

Let $\lambda$ be a partition. We write $\lambda = (\lambda_1, \cdots, \lambda_r)$ to describe each part where $\lambda_1 \geq \cdots \geq \lambda_r \geq 0$, where we allow some $\lambda_i$ to be zero. We also write $\lambda=(1^{r_1}2^{r_2}\cdots)$ when $\lambda$ consists of $r_1$ parts of 1, $r_2$ parts of 2, etc. We let $|\lambda|\colonequals \sum_{i=1}^r \lambda_i$ and also write $\lambda \vdash n$ if $|\lambda|=n$. If $\lambda, \mu$ are partitions, then we write $\lambda \supset \mu$ if $\lambda_i \geq \mu_i$ for any $i$. It is equivalent to that the Young diagram of $\lambda$ contains that of $\mu$. Also when $|\lambda|=|\mu|$, we write $\lambda \geq \mu$ if $\lambda$ is greater than equal to $\mu$ with respect to dominance order.  For a partition $\lambda=(\lambda_1, \cdots, \lambda_r)$ such that $\lambda_i \geq 1$, we let $\lambda^{(i)}$ be the partition whose parts are $\lambda_1, \cdots, \lambda_{i-1}, \lambda_i -1, \lambda_{i+1}, \cdots, \lambda_r$. It corresponds to removing a box on the $i$-th row of the Young diagram of $\lambda$ (and rearranging rows if necessary.)

Recall that there is a one-to-one correspondence between conjugacy classes of unipotent elements in $GL_n(\C)$ and partitions of $n$. The correspondence sends a conjugacy class to the partition given by taking the sizes of Jordan blocks of any unipotent element in the class. We define $u_\lambda$ to be such unipotent element corresponding to $\lambda \vdash n$, which is well-defined up to conjugation by $GL_n(\C)$. We let $H^{k}(\lambda)$ to be the $k$-th cohomology of the Springer fiber corresponding to $u_\lambda$, i.e. the set of complete flags in $\C^n$ which is stable under the action of $u_\lambda$. By Springer theory, we may regard $H^k(\lambda)$ as a $W_n$-module. Also we define $h^k(\lambda) \colonequals \dim H^k(\lambda)$ to be the $k$-th Betti number of the Springer fiber corresponding to $\lambda$.

%\section{Inequality of multiplicities}
%The monotonicity property of Kostka-Foulkes polynomials is kindly informed to the author by G. Lusztig. Here we prove the following. 
%
%\begin{prop} $\br{\chi^\lambda,H^k(\mu)} \geq \br{\chi^\lambda,H^k(\mu')}$ if $\mu \leq \mu'$.
%\end{prop}
%
%By e.g. \cite[Proposition 3.4.14 (iii)]{haiman}, it suffices to show that 
%$$\tK_{\lambda,\mu}(t) -\tK_{\lambda, \mu'}(t) \in \N[t]$$
%where
%$$\tK_{\lambda,\mu}(t)= t^{n(\mu)}K_{\lambda, \mu}(t^{-1})$$
%is a modified Kostka-Foulkes polynomial (or cocharge Kostka-Foulkes polynomial). By \cite[p. 244]{macdonald}, we have (see also \cite[(9.4)]{lu:sing})
%$$\tK_{\lambda,\mu}(t) = P_{w_\mu, w_\lambda}(t).$$
%Now the proposition follows from the following lemma.
%\begin{lem} Let $w,u, v \in \tW$ and $u\leq v$. Then $P_{u,w}(t) -P_{v,w}(t) \in \N[t]$.
%\end{lem}
%\begin{proof}\cite[Corollary 3.7]{br-mac}.
%\end{proof}

%\begin{lem} Let $\lambda, \mu$ be positive weights and $\lambda \geq \mu$. Then
%$$K_{\lambda, \mu}(q) = q^{\br{\lambda-\mu, \rho^\vee}}P_{w_\mu, w_\lambda}(q^{-1}).$$
%\end{lem}
%\begin{proof}\cite[Theorem 1.8]{kato}.
%\end{proof}

\section{Stability of dimension}
Suppose we are given $r \in \Z_{\geq2}, k\in \N$. We let $A_{k,r} = (\alpha^{k,r}_{1}, \cdots, \alpha^{k,r}_{r})$ be such that
$$\alpha^{k,r}_{1}=\alpha^{k,r}_{2} = \floor{\frac{k+r-2}{r-1}}, \alpha^{k,r}_{3} = \floor{\frac{k+r-3}{r-1}}, \alpha^{k,r}_{4} = \floor{\frac{k+r-4}{r-1}}, \cdots, \alpha^{k,r}_{r} = \floor{\frac{k}{r-1}}.$$
Thus $A_{k,r}$ is of the form
$$A_{k,r} = (a_r^{k,r}+1, \cdots, a_r^{k,r}+1, a_r^{k,r}, \cdots, a_r^{k,r})$$
with $j$ parts of $a_r^{k,r}+1$, where $0 \leq j \leq r-1$ and $j \neq 1$. Note that $A_{k,r} \supset A_{k',r}$ if $k \geq k'$.

By \cite{spal:1976}, Springer fibers for type $A$ have vanishing odd cohomologies. (Indeed, it is still true for other types.) Here we prove the following stability property of even Betti numbers of Springer fibers as follows.
\begin{thm} \label{thm:dim}For any $r \in \Z_{\geq2}, k \in \N$, there exists a polynomial $f_{k,r}(x)$ of degree $k$ with leading coefficient $\frac{1}{k!}$ which satisfies the following property; if $\lambda =(\lambda_1, \cdots, \lambda_r)$ is a partition with $\leq r$ parts such that $\lambda \supset A_{k,r}$, then $h^{2k}(\lambda) = f_{k,r}(|\lambda|)$.
\end{thm}
Thus in particular, if $\lambda, \mu$ are partitions of $\leq r$ parts such that $|\lambda|=|\mu|$ and $\lambda, \mu \supset A_{k,r}$, then $h^{2k}(\lambda) = h^{2k}(\mu)$. Note that if $\lambda = (n)$ is a partition with 1 part, then $h^{2k}(\lambda) = \delta_{k,0}$. Thus if we set $A_{k,1} \colonequals (1)$ and $f_{k,1}(x)\colonequals \delta_{k,0}$, then the theorem is true even when $r=1$ (except that for $k\neq 0$, $f_{k,1}(x)=0$ is no longer a polynomial of degree $k$).

\begin{proof} We fix $r \in \Z_{\geq 2}$ and prove the statement by induction on $k$. If $k=0$, then $h^{0}(\lambda)=1$ for any partition $\lambda$, thus we simply take $f_{0,r}(x)\colonequals1$. Now suppose that $k \geq 1$ and that the statement is true up to $k-1$. We claim that
\begin{align*}f_{k,r}(x) \colonequals& h^{2k}(A_{k,r})+\sum_{i=|A_{k,r}|}^{x-1} \left( f_{k-1,r}(i)+\cdots+f_{k-r+1,r}(i)\right) 
\\=& h^{2k}(A_{k,r})+\sum_{i=|A_{k,r}|}^{x-1} \sum_{j=k-r+1}^{k-1}f_{j,r}(x)
\end{align*}
satisfies the desired condition. (If $k <0$, we use convention that $f_{k,r}(x)\colonequals 0$.) By induction assumption it is clearly a polynomial of degree $k$ with leading coefficient $\frac{1}{k!}$.

In order to prove the statement, suppose $\lambda = (\lambda_1, \cdots, \lambda_r)$ is a partition with $\leq r$ parts such that $\lambda \supset A_{k,r}$. Then necessarily $|\lambda| \geq A_{k,r}$. If $|\lambda| = |A_{k,r}|$, then $\lambda$ should be equal to $A_{k,r}$ and $h^{2k}(\lambda) = h^{2k}(A_{k,r}) = f_{k,r}(|\lambda|)$, thus the statement is satisfied. Thus from now on we assume  $\lambda \supsetneq A_{k,r}$. Then by \cite[4.5]{fresse}, (more generally, e.g. \cite[Theorem 3.3]{g-p}, \cite[Theorem 4.4]{dk:betti}) we have
$$h^{2k}(\lambda) = h^{2k}(\lambda^{(1)})+h^{2k-2}(\lambda^{(2)})+\cdots + h^{2k-2r+2}(\lambda^{(r)}).$$
Note that $\lambda^{(i)}$ is not defined if $\lambda_i =0$. But then $i \geq 3$ since $\lambda_1, \lambda_2 \geq \alpha_{1}^{k,r} = \alpha_2^{k,r}\geq 1$ (here we use the fact that $k\geq 1$). Also as $0=\lambda_i \geq \alpha^{k,r}_i = 0$, we have $\alpha^{k,r}_i =\floor{\frac{k+r-i}{r-1}}=0,$ which means $k+1<i$. Thus we may discard $h^{2k-2i+2}(\lambda^{(i)})$ from the equation since $2k-2i+2<0$.

If $\lambda^{(i)}$ is well-defined, we claim that $\lambda^{(i)} = (\lambda^{(i)}_1, \cdots, \lambda^{(i)}_r) \supset A_{k-i+1,r}$, i.e.
$$\lambda^{(i)}_1\geq \floor{\frac{k-i+r-1}{r-1}}, \qquad \lambda^{(i)}_s\geq \floor{\frac{k-i+r-s+1}{r-1}} \quad \textup{ for } 2 \leq s \leq r.$$
If $\lambda_i \gneq \alpha_i^{k,r}$, then it is obvious as $\lambda^{(i)} \supset A_{k,r} \supset A_{k-i+1,r}$. Thus suppose otherwise, i.e. $\lambda_i= \alpha_i^{k,r}$. Recall that there exists $0 \leq j \leq r-1$, $j\neq 1$ such that $A_{k,r} = ( (\alpha_{r}^{k,r})^{r-j} (\alpha_{r}^{k,r}+1)^{j})$. If $i=1$, then $j$ cannot be zero, since otherwise $\lambda=A_{k,r}= ( (\alpha_{r}^{k,r})^r)$ which contradicts the assumption that $\lambda \supsetneq A_{k,r}$. Thus the only possibility is that $A_{k,r}=( (\alpha_{r}^{k,r})^{r-j} (\alpha_{r}^{k,r}+1)^{j})$ with $j \geq 2$ and $\lambda=((\alpha_{r}^{k,r})^{r-k}(\alpha_{r}^{k,r}+1)^{k})$ with $j<k \leq r$. Now we have $\lambda^{(1)} = ((\alpha_{r}^{k,r})^{r-k+1}(\alpha_{r}^{k,r}+1)^{k-1}) \supset A_{k,r}$ as desired.

Now suppose $i \geq 2$. Again we write $A_{k,r} = ( (\alpha_{r}^{k,r})^{r-j} (\alpha_{r}^{k,r}+1)^{j})$ for some $0 \leq j \leq r-1$. If $\lambda_i = \alpha_{r}^{k,r}+1$, then it is obvious as $A_{k,r} \supsetneq A_{k-i+1,r}$. Thus suppose otherwise, i.e. $\lambda_i = \alpha_{r}^{k,r}$. But then $i>j$, and we have $A_{k-i+1,r}=( (\alpha_{r}^{k,r}-1)^{i-j} (\alpha_{r}^{k,r})^{r-i+j})$. Since $\lambda_i = \cdots = \lambda_r = \alpha_{r}^{k,r}$ as we assume that $\lambda \supset A_{k,r}$, this case is also clear.

Therefore, by induction hypothesis
$$h^{2k}(\lambda) = h^{2k}(\lambda^{(1)})+f_{k-1,r}(|\lambda|-1)+\cdots+f_{k-r+1,r}(|\lambda|-1).$$
As $\lambda^{(1)} \supset A_{k,r}$, we can iterate this process until we reach $A_{k,r}$. Then we have
$$h^{2k}(\lambda) = h^{2k}(A_{k,r})+\sum_{i=|A_{k,r}|}^{|\lambda|-1} \left(f_{k-1,r}(i)+\cdots+f_{k-r+1,r}(i)\right)= f_{k,r}(|\lambda|).$$
The theorem is proved.
\end{proof}

Suppose $k <r$. Then $A_{k,r}=(1^{k+1})$ and by following the proof above we see that $f_{k,r}(x)$ does not depend on $r$. Thus $f_k(x) \colonequals \lim_{r \rightarrow \infty} f_{k,r}(x)$ is well-defined. Now we have the following corollary.
\begin{cor} \label{cor:rinf} Suppose $\lambda=(\lambda_1, \cdots, \lambda_r) \vdash n$ be a partition such that $\lambda_1, \cdots, \lambda_{k+1} \geq 1$. Then $h^{2k}(\lambda)$ is equal to $f_{k}(|\lambda|)$. Also it is the same as $h^{2k}(\B_n)$.
\end{cor}
\begin{proof} The first part is obvious. The second part follows by setting $\lambda=(1^n)$.
\end{proof}

\begin{example} Here are some examples of $f_{k,r}(x)$. Also note that $f_{k,2}(x)$ is calculated in \cite[Proposition 8.1]{dk:betti}.
$$
\begin{array}{|c|c|c|c|c|c|}
\hline
r\diagdown k&0&1&2&3&4\\
\hline
2&1& x-1 & \frac{1}{2} \left(x^2-3 x\right) & \frac{1}{6} \left(x^3-6 x^2+5 x\right) & \frac{1}{24} \left(x^4-10 x^3+23 x^2-14 x\right) \\
\hline
3&1& x-1 & \frac{1}{2} \left(x^2-x-2\right) & \frac{1}{6} \left(x^3-13 x+6\right) & \frac{1}{24} \left(x^4+2 x^3-37 x^2+34 x\right) \\
\hline
4&1&x-1 & \frac{1}{2} \left(x^2-x-2\right) & \frac{1}{6} \left(x^3-7 x\right) & \frac{1}{24} \left(x^4+2 x^3-13 x^2-38 x+24\right) \\
\hline
5& 1&x-1 & \frac{1}{2} \left(x^2-x-2\right) & \frac{1}{6} \left(x^3-7 x\right) & \frac{1}{24} \left(x^4+2 x^3-13 x^2-14 x\right) \\
\hline
6&1& x-1 & \frac{1}{2} \left(x^2-x-2\right) & \frac{1}{6} \left(x^3-7 x\right) & \frac{1}{24} \left(x^4+2 x^3-13 x^2-14 x\right) \\
\hline
\cdots&\cdots &\cdots& \cdots& \cdots&\cdots\\
\hline
\end{array}
$$
$$
\hspace*{-0.3in}
\begin{array}{|c|c|c|}
\hline
r\diagdown k & 5& 6\\
\hline
2& \frac{1}{120} \left(x^5-15 x^4+65 x^3-105 x^2+54 x\right) & \frac{1}{720} \left(x^6-21 x^5+145 x^4-435 x^3+574 x^2-264 x\right) \\
\hline
3& \frac{1}{120} \left(x^5+5 x^4-75 x^3+55 x^2+134 x\right) & \frac{1}{720} \left(x^6+9 x^5-125 x^4-45 x^3+1204 x^2-1044 x\right) \\
\hline
4& \frac{1}{120} \left(x^5+5 x^4-15 x^3-185 x^2+194 x+120\right) & \frac{1}{720} \left(x^6+9 x^5-5 x^4-525 x^3+364 x^2+1596 x-720\right) \\
\hline
5& \frac{1}{120} \left(x^5+5 x^4-15 x^3-65 x^2-166 x+240\right) & \frac{1}{720} \left(x^6+9 x^5-5 x^4-165 x^3-1076 x^2+1956 x\right) \\
\hline
6& \frac{1}{120} \left(x^5+5 x^4-15 x^3-65 x^2-46 x+120\right) & \frac{1}{720} \left(x^6+9 x^5-5 x^4-165 x^3-356 x^2-204 x+720\right) \\
\hline
7& \frac{1}{120} \left(x^5+5 x^4-15 x^3-65 x^2-46 x+120\right) & \frac{1}{720} \left(x^6+9 x^5-5 x^4-165 x^3-356 x^2+516 x\right) \\
\hline
8& \frac{1}{120} \left(x^5+5 x^4-15 x^3-65 x^2-46 x+120\right) & \frac{1}{720} \left(x^6+9 x^5-5 x^4-165 x^3-356 x^2+516 x\right) \\
\hline
\cdots& \cdots & \cdots\\
\hline
\end{array}
$$
$$
\begin{array}{|c|c|}
\hline
r\diagdown k & 7\\\hline
2& \frac{1}{5040}(x^7-28 x^6+280 x^5-1330 x^4+3199 x^3-3682 x^2+1560 x )\\\hline
3& \frac{1}{5040}( x^7+14 x^6-182 x^5-490 x^4+5089 x^3-7084 x^2+2652 x )\\\hline
4& \frac{1}{5040}( x^7+14 x^6+28 x^5-1120 x^4-581 x^3+8666 x^2-1968 x )\\\hline
5& \frac{1}{5040}( x^7+14 x^6+28 x^5-280 x^4-3941 x^3+5306 x^2+8952 x )\\\hline
6& \frac{1}{5040}( x^7+14 x^6+28 x^5-280 x^4-1421 x^3-4774 x^2+11472 x+5040) \\\hline
7& \frac{1}{5040}( x^7+14 x^6+28 x^5-280 x^4-1421 x^3+266 x^2-3648 x+10080 )\\\hline
8& \frac{1}{5040}( x^7+14 x^6+28 x^5-280 x^4-1421 x^3+266 x^2+1392 x+5040 )\\\hline
9& \frac{1}{5040}( x^7+14 x^6+28 x^5-280 x^4-1421 x^3+266 x^2+1392 x+5040 )\\\hline
\cdots&\cdots\\\hline
\end{array}
$$
\end{example}

\section{Main theorem}
In the previous section we proved stability of Betti numbers of Springer fibers. Together with the lemma stated below, we may obtain the stability of Springer representations on each degree. 
\begin{lem} \label{lem:mono} Let $\lambda$, $\lambda'$, be partitions such that $|\lambda|=|\lambda'|$. If $\lambda \geq \lambda'$, then there is a injective homomorphism $H^{2k}(\lambda) \hookrightarrow H^{2k}(\lambda')$ of $W_{|\lambda|}$-modules.
\end{lem}
\begin{proof} See e.g. \cite[Section 4]{g-p}
\end{proof}
\begin{rmk} As noted in \cite[Section 4]{g-p}, this property is directly related to monotonicity of Kostka-Foulkes polynomials. Indeed, by \cite{lu:green} Kostka-Foulkes polynomials (with some modification) are special kinds of Kazhdan-Lusztig polynomials of the extended affine Weyl group for type $A$, whose monotonicity is proved in \cite[Corollary 3.7]{br-mac}.
\end{rmk}

\begin{thm} Let $\lambda$ and $\lambda'$ be partitions with $|\lambda|=|\lambda'|$. If there exists $r$ such that $\lambda$ and $\lambda'$ consist of $\leq r$ parts and $\lambda, \lambda' \supset A_{k,r}$, then $H^{2k}(\lambda)$ and $H^{2k}(\lambda')$ are the same $W_{|\lambda|}$-representation of dimension $f_{k,r}(|\lambda|)$, where $f_{k, r}(x)$ is defined in Theorem \ref{thm:dim}.
\end{thm}
\begin{proof} By Theorem \ref{thm:dim}, $h^{2k}(\lambda)=h^{2k}(\lambda')=f_{k,r}(|\lambda|)$. Thus if $\lambda$ and $\lambda'$ are comparable with respect to dominance order, then the result directly follows from Lemma \ref{lem:mono}. However, we can always find $\lambda_{max}$ such that $|\lambda_{max}|=|\lambda|=|\lambda'|$, $\lambda_{max} \supset A_{k,r}$, and $\lambda_{max} \geq \lambda, \lambda'$ by setting
$$\lambda_{max} \colonequals (|\lambda|-A_{k,r}+\alpha_{1}^{k,r}, \alpha_{2}^{k,r}, \cdots, \alpha_r^{k,r}).$$
In other words, $\lambda_{max}$ is obtained from $A_{k,r}$ by attaching $|\lambda|-|A_{k,r}|$ number of boxes to the first row of the Young diagram of $A_{k,r}$. Thus the result still follows from Lemma \ref{lem:mono}.
\end{proof}
If $r\rightarrow \infty$, then we have the following corollary.
\begin{cor} \label{cor:strinf} Let $\lambda$ be partitions of $\geq k+1$ parts. Then $H^{2k}(\lambda)$ and $H^{2k}(\B_{|\lambda|})$ are the same  $W_{|\lambda|}$-modules of degree $f_k(x)$, where $f_k(x) \colonequals \lim_{r \rightarrow \infty} f_{k,r}(x)$.
\end{cor}
\begin{proof} It follows from Corollary \ref{cor:rinf} and Lemma \ref{lem:mono}.
\end{proof}

\begin{rmk} \cite[Theorem 7.1]{c-f} proved that $\{H^k(\B_n))\}_{n\in \N}$ is (uniformly) representation stable for each $k$. Now consider a sequence of partitions $\lambda^1 \subset \lambda^2 \subset \cdots$ which has eventually $\geq k+1$ parts, i.e. there exists $N \in \N$ such that $\lambda_N$ has $\geq k+1$ parts. Then Corollary \ref{cor:strinf} implies that the sequence $\{H^{2k}(\lambda^n)\}_{n\in\N}$ is eventually the same as $\{H^{2k}(\B_n)\}_{n \in \N}$ as representations of symmetric groups. Thus in this case we recover the result of \cite{mt}.
\end{rmk}

\bibliographystyle{amsalphacopy}
\bibliography{stable}

\end{document}